\documentclass{article}
\usepackage[utf8]{inputenc}
\usepackage{enumerate}
\usepackage{amssymb, bm}
\usepackage{amsthm}
\usepackage{amsmath}
\usepackage{tikz-cd} 
\usepackage{bbold}
\usepackage[english,francais]{babel}
 \usepackage[all]{xy}
\title{Albanese morphism of log smooth klt compact K\"ahler manifold with nef log anticanonical divisor}
\author{Xiaojun WU}
\date{\today}
\newtheorem{mythm}{Theorem}

\newtheorem{myprop}{Proposition}
\newtheorem{mycor}{Corollary}
\newtheorem{mydef}{Definition}

\newtheorem{myex}{Example}
\begin{document}
\def\cI{\mathcal{I}}
\def\Z{\mathbb{Z}}
\def\Q{\mathbb{Q}}  \def\C{\mathbb{C}}
 \def\R{\mathbb{R}}
 \def\N{\mathbb{N}}
 \def\H{\mathbb{H}}
  \def\P{\mathbb{P}}
 \def\rC{\mathcal{C}}
  \def\nd{\mathrm{nd}}
  \def\d{\partial}
 \def\dbar{{\overline{\partial}}}
\def\dzbar{{\overline{dz}}}
 \def\ii{\mathrm{i}}
  \def\d{\partial}
 \def\dbar{{\overline{\partial}}}
\def\dzbar{{\overline{dz}}}
\def \ddbar {\partial \overline{\partial}}
\def\cN{\mathcal{N}}
\def\cE{\mathcal{E}}  \def\cO{\mathcal{O}}
\def\cF{\mathcal{F}}
\def\P{\mathbb{P}}
\def\cJ{\mathcal{J}}
\def \loc{\mathrm{loc}}
\def \log{\mathrm{log}}
\def \det{\mathrm{det}}
\def \cC{\mathcal{C}}
\bibliographystyle{plain}
\def \dim{\mathrm{dim}}
\def \RHS{\mathrm{RHS}}
\def \liminf{\mathrm{liminf}}
\def \ker{\mathrm{Ker}}
\def \Pic{\mathrm{Pic}}
\def \Alb{\mathrm{Alb}}
\def \tors{\mathrm{Tors}}
\def \ch{\mathrm{ch}}
\def \Id{\mathrm{Id}}
\def \td{\mathrm{td}}
\def \id{\mathrm{id}}
\def \D{\mathbb{D}}
\maketitle
\begin{abstract}
Let $(X, \omega)$ be an n-dimensional compact K\"ahler manifold.
Let $D=\sum (1-\beta_j) Y_j=\sum (1-\beta_j) [s_j=0]$ a divisor with simple normal crossings with $\beta_j \in ]0,1[$ such that $-(K_X+D)$ is nef. 
We show that its Albanese map is submersion outside an analytic set of codimension larger than two with connected fibres.
\end{abstract}
\section{Regularity of Monge-Ampère equation with conic singularities}
The following result on the regularities of the Monge-Ampère equation is well-known to experts.
Since we have not found a precise reference of some result, we provide some proof here.
Let us fix some notations.
Let $(X,D)$ be a \textit{log smooth klt pair}, i.e. $X$ is a compact Kähler manifold, and $D= \sum (1-\beta_k)Y_k$ is a $\R$-divisor with simple normal crossing support such that $\beta_k\in ]0,1[$ for all $k$.
The coefficients $(\beta_k)$ are called standard orbifold is $
\beta_k=1/n_k$ for some $n_k \in \N^*$ for all $k$.
We denote by $\omega_\beta$ the standard cone metric attached with $(\C^n, D)$, i.e.
$$\omega_\beta:= i\sum_{k=1}^d \frac{ dz_k \wedge d \bar z_k}{|z_k|^{2(1-\beta_k)}} +i\sum_{k=d+1}^n dz_k \wedge d \bar z_k.  $$
A metric $\omega$ is said to have conic singularities if it is quasi-isometric to the model metric with conic singularities: more precisely, near each point $p\in \mathrm{Supp}(D)$ where $\mathrm{Supp}(D)$ is defined by the equation $\{z_1 \cdots z_d=0\}$ for some holomorphic system of coordinates $(z_i)$, we want $\omega$ to satisfy
$$C^{-1} \omega_{\beta} \le \omega \le C \omega_{\beta}$$
for some constant $C>0$ and some $
\beta$ (near $p$).
We have the following result due to \cite{CGP}.
\begin{myprop}
 Let $X$ be a compact Kähler manifold and $D=\sum (1-\beta_j) Y_j=\sum (1-\beta_j) [s_j=0]$ a divisor with simple normal crossings with $\beta_j \in ]0,1[$. Let $\mu_{D}=dV_\omega/\prod_j |s_j|^{2(1-\beta_j)}$ be a volume form with conic singularities along $D$, $\mu \in \R$, and $\omega$ a Kähler form on $X$. Then any (bounded) solution $\varphi$ of
\[(\omega+ i\d \dbar \varphi)^n = e^{\mu \varphi} \mu_{D}\]
is Hölder-continuous and the metric $\omega+ \id\dbar \varphi$ has conic singularities along $D$.
\end{myprop}
In the following section, we start by considering the standard orbifold case to illustrate the proof. 
In this case, we have a more precise regularity of $\varphi$, and we adapt the notations of \cite{CGP} in this case.
As written at the beginning of this section, the following Proposition 3 is well-known to experts.

We fix $\D^n \subset \C^n$ the unit polydisk centered at the origin, and a divisor $D= \sum_{k=1}^d (1-1/n_k) D_k$ where $D_k=\{z_k=0\}$ for all $k$, and $d\le n$. 

Consider the branched cover
\[
\begin{array}{cccc}
\pi:& \mathbb{D}^d \times \D ^{n-d} & \longrightarrow & \mathbb{D}^d \times \D ^{n-d} \\
&(z_1, \ldots, z_d, z_{d+1}, \ldots, z_n) & \longmapsto & (z_1^{n_1}, \ldots, z_d^{n_d}, z_{d+1}, \ldots, z_n).
\end{array}
\]
If we denote by $w$ the coordinates upstairs, 
$$\pi^* \omega_\beta=i\sum_{k=1}^d |n_k|^2 dw_k \wedge d \bar w_k +i\sum_{k=d+1}^n dw_k \wedge d \bar w_k.$$
In particular, $\pi^* \omega_\beta$ is a genuine metric upstairs.
Recall that in the standard orbifold case, a function $f$ is said to have orbifold $ C^{k,\alpha}$ regularity (denoted by $C^{k, \alpha, \beta}$) if its pull-back $\pi^*f$ by the ramified cover $\pi$ is $ C^{k,\alpha}$ in the usual sense. 
More generally, let $\tau$ (resp.  $\sigma$) be a bounded $(1,0)$-form (resp. $(1,1)$-form) on $\D^n\setminus D$. Then we say that $ \tau \in C^{\alpha, \beta}  $ if for all $k$, we have $ \pi^*\tau(\frac{\d}{\d w_k}) \in  C^{\alpha}$ in the usual sense (resp.
 $ \sigma \in C^{\alpha, \beta}  $ if for all $k$ and $l$, we have $\pi^*\sigma(\frac{\d}{\d w_k},\frac{\d}{\d \overline{w}_l}) \in  C^{\alpha}$ in the usual sense).
 Define
 $$C^{2,\alpha, \beta}:=\{f\in L^{\infty}(\D^n);  f,\d f, \dbar \d f \in C^{\alpha,\beta} \}$$
 with the associated natural norm.
 It is checked in Lemma 7.2 \cite{CGP} that this definition coincides with the definition of Donaldson \cite{Don}.
 
Cover $X$ be open sets $U_i$ so that each open set is biholomorphic to the unit polydisk with ramified cover associated with $D$.
Denote $\tilde{U}_i$ the ramified cover of $\pi_i : \tilde{U_i} \to U_i$.
Let $N_i$ be the ramified order.
We have a more precise regularity result than Proposition 1, also stated in \cite{CGP}.
\begin{myprop}
Let $(X, D)$ as in Proposition 1 with standard orbifold coefficients, and let $\varphi\in L^{\infty}(X)$ be any solution of 
\[(\omega+ i\d \dbar \varphi)^n = e^{\mu \varphi} \mu_{D}\]
Then $\varphi$ belongs to the class $C^{2,\alpha, \beta}$. 
\end{myprop}

More precisely, in the standard orbifold case, by elliptic regularity, one can show that $\pi^* \varphi$ is in fact smooth.
\begin{myprop}
Let $(X, D)$ as in Proposition 1 with standard orbifold coefficients, and let $\varphi\in L^{\infty}(X)$ be any solution of 
\[(\omega+ i\d \dbar \varphi)^n = e^{\mu \varphi} \mu_{D}\]
Then $\pi^* \varphi$ is smooth in each ramified cover.
\end{myprop}
\begin{proof}
Note that $\pi^* d \nu_D$ is a smooth volume from upstairs.
Without loss of generality, assume that $\omega= i \d \dbar \psi$ for some smooth function locally.
Locally the pullback of the Monge-Ampère equation on the ramified cover can be written as
$$\det(\frac{\d^2}{\d w_i 
\d \overline{w}_j} \pi^*(\psi+\varphi))=ge^{\mu \pi^* \varphi}$$
for some nowhere vanishing smooth function $g$.
By Proposition 2, $\pi^* (\omega+ i\d \dbar \varphi) $ is a genuine metric with $C^{0, \alpha}$ coefficients.
Denote locally
$$\pi^* (\omega+ i\d \dbar \varphi)=\sum_{i,j} h_{i j}\sqrt{-1} dw^i \wedge d  \overline{w}^j.$$
From the Monge-Ampère equation, we have an a priori positive lower bound on the determinant of $h_{ij}$. 
In particular, if $h_{ij} \in C^{k,\alpha}$ for some $k \geq 0$, then
$\det(h_{ij})^{-1} \in C^{k, \alpha}$ (hence $h^{ij} \in C^{k, \alpha}$).
Take $\frac{\d}{\d w_l} \log (\forall l)$ for both sides of equation
$$h^{ij} \frac{\d^2}{\d w_i 
\d \overline{w}_j} \frac{\d}{\d w_l} \pi^* (\omega+ i\d \dbar \varphi)= \frac{\d}{\d w_l} \log g + \mu \frac{\d}{\d w_l} \pi^* \varphi. $$
Inductively for $k \geq 0$, by interior Schauder estimates for linear elliptic equations with coefficients in $C^{k, \alpha}$,
we obtain the a priori $C^{k+3, \alpha}$
norm estimate of $\pi^* \varphi$ in terms of the $C^{k+1, \alpha}$ norm of $\pi^* \varphi$. 
Thus $\pi^* \varphi$ is smooth in any ramified cover by bootstrapping.
\end{proof}

We now recall the definition of a singular metric on a vector bundle according to \cite{Paun16}.
\begin{mydef}
A singular Hermitian metric $h$ on $E$ is given locally by a measurable, possibly unbounded map with values in the set of semi-positive Hermitian matrices, such that $0<\det h < \infty$ almost everywhere.
\end{mydef}
By this definition, a solution of the Monge-Ampère equation with conic singularities defines a singular metric on $T_X$.
In particular, the solution also induces a singular metric on any quotient bundle of $T_X$.
We observe that by the Monge-Amp\`ere equation, the Ricci curvature of the singular metric is well defined as a current.
However, one can notice that the curvature tensor of $T_X$ is not necessarily well-defined as a current with values in semi-positive, possibly unbounded Hermitian matrices.

Now we return to the existence and regularity of the Monge-Amp\`ere equation for the general coefficient case.

In fact, the work of \cite{Gue} and \cite{CGP} gives the following weak estimate for the following type of Monge-Amp\`ere equation.
The theorem is not essentially used for the following section, but the discussion after the theorem also applies to this slight generalisation.
So we still state it explicitly.
\begin{mythm}
Let $X$ be an n-dimensional compact K\"ahler manifold, and let $D = \sum_i
a_i D_i$, $E = \sum_j
b_j E_j$ be two effective
$\R$-divisors with simple normal crossing support, such that for all $1 \leq i \leq r$, $0 < a_i < 1$.
Assume that $D$ and $E$ have no common irreducible component.
Let $\omega$ be a K\"ahler metric on $X$, $dV$ a smooth volume form, and let $\varepsilon>0$. Then the weak
solution of the Monge-Amp\`ere equation
$$\langle (\omega+\frac{i}{2 \pi} \d \dbar \varphi)^n \rangle=e^{\varepsilon \varphi} \frac{\prod|t_j|^{2 b_j}dV}{\prod|s_i|^{2 a_i}}$$
exists, which is smooth on $X \smallsetminus (D \cup E)$ and has an upper bound by a metric with conic singularity along $D$.
Here $\langle \bullet \rangle$ is the positive intersection product defined in {\rm \cite{BEGZ}}.
Here $s_i$(resp. $t_j$) is the canonical section of $\cO(D_i)$ (resp. $\cO(E_j)$) and $|s_i|^2$ (resp. $|t_j|^2$) is the norm of $s_i$ (resp. $t_j$) with respect to some smooth metric.
\end{mythm}
We observe that the existence of a solution is proved in \cite{BEGZ}.
As a consequence of Theorem 1, there exists $C >0$ such that the solution has on $X \smallsetminus (D \cup E)$ an upper bound 
$$ \omega+ \frac{i}{2 \pi} \d \dbar \varphi \leq  \frac{C  \omega}{ \prod_i |s_i|^{ 2a_i} }. $$
By the Monge-Amp\`ere equation, we find on $X \smallsetminus (D \cup E)$ a  lower bound 
$$ \omega+ \frac{i}{2 \pi} \d \dbar \varphi \geq  e^{\varepsilon \varphi} \frac{\prod|t_j|^{2 b_j}\omega}{\prod|s_i|^{2 a_i}} (\frac{C  }{ \prod_i |s_i|^{ 2a_i} })^{-(n-1)}. $$
Notice that since the solution is smooth on $X \smallsetminus (D \cup E)$, the above inequalities are satisfied pointwise.
By the result of \cite{BEGZ}, $|\varphi|$ is uniformly bounded on $X$. 
In particular, we have
$$ \omega+ \frac{i}{2 \pi} \d \dbar \varphi \geq   \frac{C\prod|t_j|^{2 b_j}\omega}{\prod|s_i|^{2 a_i}} (\frac{C  }{ \prod_i |s_i|^{ 2a_i} })^{-(n-1)}. $$
In conclusion outside $D \cup E$, the solution $\omega+\frac{i}{2 \pi} \d \dbar \varphi$ viewed as a Hermitian form over $T_X$ with respect to $\omega$ has positive eigenvalues bounded from above by $\frac{C  }{ \prod_i |s_i|^{ 2a_i} }$ and bounded from below by $\frac{C\prod|t_j|^{2 b_j}}{\prod|s_i|^{2 a_i}} (\frac{C  }{ \prod_i |s_i|^{ 2a_i} })^{-(n-1)}. $

Let us observe that for the singular metric on the determinant line bundle of the quotient bundle $Q$ given by a short exact sequence of vector bundles 
$$0 \to S \to T_X \to Q \to 0,$$
the curvature form is well-defined as a current. We detail the argument below.
Suppose we are in the situation of Theorem 1, with the same notation above.
Since the metric is smooth outside $D\cup E$, we only need to study the neighbourhood of $D \cup E$.
By a $C^{\infty}$ splitting of the exact sequence, we can view $Q$ as a subbundle of $T_X$. 
$\omega+\frac{i}{2 \pi} \d \dbar \varphi$ thus induces  a Hermitian form over $Q$ which we will denote by $\omega+\frac{i}{2 \pi} \d \dbar \varphi|_Q$.
By the minimax principle, for the induced Hermitian form on $Q$,
the eigenvalues are bounded from above by $\frac{C  }{ \prod_i |s_i|^{ 2a_i} }$ and bounded from below by $\frac{C\prod|t_j|^{2 b_j}}{\prod|s_i|^{2 a_i}} (\frac{C  }{ \prod_i |s_i|^{ 2a_i} })^{-(n-1)}. $
To prove that the curvature of $\det(Q)$ is well-defined as a current (not necessarily positive), it is enough to prove that $\log(\det(\omega+\frac{i}{2 \pi} \d \dbar \varphi|_Q)) \in L^1_{\loc}$.
$\det(\omega+\frac{i}{2 \pi} \d \dbar \varphi|_Q)$ is the product of all eigenvalues of the Hermitian form 
$\omega+\frac{i}{2 \pi} \d \dbar \varphi|_Q$.
Thus we get the estimate  for the potential
$$|\log(\det(\omega+\frac{i}{2 \pi} \d \dbar \varphi|_Q))| \leq \sum_i C_i \log|s_i|^2+ \sum_j C_j \log|t_j|^2+C$$
for some $C_i >0$, $C_j >0$ and $C >0$.
In the following, we will refer to this type of control as potentials possessing at most logarithmic poles along $D \cup E$.
Notice also that for any $i$, $\log|z_i|$ is locally integrable with respect to the euclidean metric.
In particular, the curvature of the induced metric on $\det(Q)$ is well defined as a current since locally it is the $i\d \dbar$ of some $L^1_{\loc}$ function.

For any global potential $\psi$ (defined on $X$) possessing at most logarithmic poles along $D \cup E$.
As analogy to Monge-Amp\`ere operator in the sense of Bedford-Taylor \cite{BT}, we want to define
$i \d \dbar \psi \wedge \omega_{\varphi}^{n-1}$.
Fix $a \in [0,1[$, 
the integration
$$\int \log(r) r^{1-2a}=\frac{1}{2-2 a} (\log(r)-\frac{1}{2-2 a})r^{2-2 a}$$
is finite over $[0,1]$.
In particular,
$\log \vert z \vert \vert z \vert^{-2a}$ as a function $z \in \C$ is locally integrable near 0 with respect to the  Lebesgue measure.
Thus the coefficients of
$\psi \wedge \omega_{\varphi}^{n-1}$
are locally integrable with respect to the Lebesgue measure.
Define 
$$i \d \dbar \psi \wedge \omega_{\varphi}^{n-1}:= i \d\dbar ( \psi \wedge \omega_{\varphi}^{n-1}).$$
By Stokes theorem, we have that
$$\int_X i \d \dbar \psi \wedge \omega_{\varphi}^{n-1}=0.$$
\section{Albanese morphism}
In this section, we start by generalising the results of \cite{Cao13} to log smooth cases with standard orbifold coefficients. 
To start with, we need the following result.
\begin{mythm}
Let $(X, \omega)$ be an n-dimensional compact K\"ahler manifold.
Let $D=\sum (1-1/n_j) Y_j=\sum (1-1/n_j) [s_j=0]$ a divisor with simple normal crossings with $n_j \in \N^*$ such that $-(K_X+D)$ is nef. 
Let
$$0 =\cE_0 \subset \cE_1 \subset \cdots \subset \cE_s=T_X$$ be a filtration of torsion-free subsheaves such that $\cE_{i+1}/\cE_i$ is an $\omega$-stable torsion-free subsheaf of $T_X/\cE_i$ of maximal slope. 
Then for any $i$, the slope of $\cE_{i+1}/\cE_i$ with respect to $\omega^{n-1}$,
namely
$$\mu(\cE_{i+1}/\cE_i):= \int_X c_1(\cE_{i+1}/\cE_i) \wedge \omega^{n-1},$$
is positive.
\end{mythm}
\begin{proof}
By the stability condition, to prove the theorem, it is sufficient to prove that for any $i$
$$\int_X c_1(T_X/\cE_i)\wedge \omega^{n-1} \geq 0.$$
Let $r$ be the generic rank of $T_X/\cE_i$ for some fixed $i$.
The naturally induced morphism $\wedge^r T_X \to \det(T_X/\cE_i)$ corresponds to a section $\tau \in H^0(X, \det(T_X/\cE_i) \otimes \Omega^r_X )$.
$\tau$ is non-vanishing outside a closed analytic set of codimension at least two on which $T_X/\cE_i$ is locally free.
Fix an arbitrary smooth metric $h$ on $\det(T_X/ \cE_i)$.

The critical step is the existence of positive closed $(1,1)$-current in a K\"ahler class which is smooth outside an SNC divisor and whose Ricci curvature can be taken ``arbitrary small" outside the divisor using the theorems in \cite{GP} and \cite{CGP}.

Let $T_\varepsilon$ be a sequence of smooth forms in $c_1(-(K_X+D))$ such that $T_\varepsilon \geq - \varepsilon \omega$ whose existence is ensured by the nefness condition.
To get the lower bound, we want to solve the following K\"ahler-Einstein type of equation
$$\mathrm{Ric}(\omega_{\varphi_{\varepsilon}})=-\varepsilon \omega_{\varphi_{\varepsilon}} + \varepsilon \omega + T_{\varepsilon} +[D]$$
where $\omega_{\varphi_{\varepsilon}}:= \omega+i \d \dbar \varphi_{\varepsilon}$ is the unknown.
Notice that both sides belong to the class $c_1(-K_{X})$.
In order to solve the K\"ahler-Einstein type of equation, we thus solve the following Monge-Amp\`ere equation by Theorem 1.
Let $\gamma$ be a smooth representative of the class $\{[D]\}$, which is induced from the curvature forms of some smooth metrics $(\cO(Y_i),h_i)$.
By the $\d \dbar$-lemma, there exists $f_{\varepsilon} \in C^{\infty}(X)$ such that $T_{\varepsilon}+\gamma=\mathrm{Ric}(\omega)+\frac{i}{2\pi} \d \dbar f_{\varepsilon}$.
The Monge-Amp\`ere equation equivalent to the K\"ahler-Einstein type of equation can be written as
$$\omega_{ \varphi_{\varepsilon}}^n=\frac{\omega^n e^{\varepsilon \varphi_\epsilon-f_{\varepsilon}}}{\prod_i |s_i|_{h_i}^{2(1-1/n_i)}}.$$

Take the same notations as after Proposition 1.
Note that $\pi^* \omega_{ \varphi_{\varepsilon}}$ is a smooth metric when pulling back onto the ramified cover by Proposition 3.
Calculate on the ramified cover
$$i \d \dbar \pi^* \log \vert \tau \vert_{\omega_{\varphi_\epsilon,h}}^2= \frac{i \{D' \pi^*  \tau , D' \pi^* \tau \}}{\pi^* \vert \tau \vert^2}- \frac{i \{D' \pi^*  \tau ,  \pi^* \tau \}\wedge \{ \pi^*  \tau , D' \pi^* \tau \}}{\pi^* \vert \tau \vert^4}$$
$$-\pi^* i \Theta(\det(T_X/\cE_i),h)- \frac{\{ \pi^*  \tau , i \Theta(\pi^* \omega_{\varphi_\epsilon}) \pi^* \tau \}}{\pi^* \vert \tau \vert^2}$$
where $\{ \}$ is a canonical sesquilinear pairing
$$C^\infty(\wedge^p T^*_X \otimes \Omega_X^r \otimes \det(T_X/\cE_i))\times C^\infty(\wedge^q T^*_X \otimes \Omega_X^r \otimes \det(T_X/\cE_i)) \to C^\infty(\wedge^{p+q} T^*_X ).$$
Note that the right-handed term on the first line is positive in the sense of currents by Schwarz inequality.
Observe that the right-handed term is locally integrable (and hence well defines its product with $\pi^* \omega_{\varphi_\epsilon}^{n-1}$ in the sense of currents on the ramified cover).
Since $\pi^* \omega_{\varphi_\epsilon}$ is smooth, 
the coefficients of
$\frac{\{ \pi^*  \tau , i \Theta(\pi^* \omega_{\varphi_\epsilon}) \pi^* \tau \}}{\pi^* \vert \tau \vert^2}$ is locally bounded.
On the other hand,
$\vert D' \pi^* \tau \vert^2/\vert \pi^* \tau \vert^2 $ is locally integrable with respect to the Lebesgue measure.
To see this, let $\cI$ be the (local) ideal defined by the coefficients of $\pi^* \tau$.
Without loss of generality, it is enough to consider the case that $\cI$ is non-trivial.
By Hironaka's resolution of singularities \cite{Hir64}, there exists a modification $p: U' \to \tilde{U} $ such that $\cI \cdot \cO_{U'}=\cO(- \sum_i \lambda_i E_i)$ with $\sum_i \lambda_i E_i$ an effective SNC divisor.
Note that since $\tau$ is non-vanishing outside a codimension 2 set, $p$ is not an identity map.
Denote by $D \cI$ the (local) ideal defined by the differentials of the coefficients of $\pi^* \tau$.
Then $D \cI \cdot \cO_{U'}$ is contained in $\cO(- \sum_i (\lambda_i-1) E_i)$.
$K_{U'/\tilde{U}}=\sum_i \nu_i E_i$ with $\nu_i \geq 1$.
In particular, 
the pullback of Lebesgue measure can be written as $\prod_i \vert s_{E_i} \vert^{2 \nu_i}$ times a smooth nowhere vanishing form.
Thus the coefficient of the product of $p^* (\vert D' \pi^* \tau \vert^2/\vert \pi^* \tau \vert^2)$ with the pullback of Lebesgue measure is locally bounded on $U'$.
In particular, $\vert D' \pi^* \tau \vert^2/\vert \pi^* \tau \vert^2 $ is locally integrable.

On the ramified cover, the Chern curvature, which is also the Levi-Civita curvature, satisfies
$$\mathrm{Ric}(\pi^* \omega_{\varphi_{\varepsilon}})=-\varepsilon \pi^* \omega_{\varphi_{\varepsilon}}+\pi^* (T_\epsilon+\epsilon \omega) \geq -\varepsilon \pi^* \omega_{\varphi_{\varepsilon}}.$$
Consider
$$i \d \dbar \pi^* \log \vert \tau \vert_{\omega_{\varphi_\epsilon,h}}^2 \wedge \pi^* \omega_{\varphi_\epsilon}^{n-1} \geq  (-\pi^* i \Theta(\det(T_X/\cE_i),h)- \frac{\{ \pi^*  \tau , i \Theta(\pi^* \omega_{\varphi_\epsilon}) \pi^* \tau \}}{\pi^* \vert \tau \vert^2}) \wedge \pi^* \omega_{\varphi_\epsilon}^{n-1}.$$
On the other hand, by Proposition 2.7 \cite{Cao13}, with local curvature calculations, we have
$$\frac{\{ \pi^*  \tau , i \Theta(\pi^* \omega_{\varphi_\epsilon}) \pi^* \tau \}}{\pi^* \vert \tau \vert^2} \wedge \pi^* \omega_{\varphi_\epsilon}^{n-1} \leq \epsilon \pi^* \omega_{\varphi_\epsilon}^n.$$
Note that 
$\log \vert \tau \vert_{\omega_{\varphi_\epsilon,h}}^2$
possesses at most logarithmic poles along $D$.
In fact, the dual metric on $T_X^*$ could only vanish along $D$ when the initial metric on $T_X$ has conic singularities.
In particular,
$i \d \dbar \log \vert \tau \vert_{\omega_{\varphi_\epsilon,h}}^2 \wedge \omega_{\varphi_\epsilon}^{n-1}$
is well-defined
as discussed after Theorem 1.

Let $\theta_i$ be a partition of unity associated with the chosen open cover.
$$0=\sum_i \int_X \theta_i i \d \dbar \log \vert \tau \vert_{\omega_{\varphi_\epsilon,h}}^2 \wedge \omega_{\varphi_\epsilon}^{n-1}=\sum_i \int_X i \d \dbar \theta_i  \log \vert \tau \vert_{\omega_{\varphi_\epsilon,h}}^2 \wedge \omega_{\varphi_\epsilon}^{n-1}$$
$$=\sum_i \int_{\tilde{U}_i}  \frac{1}{N_i} \pi_i^* (i \d \dbar\theta_i   \log \vert \tau \vert_{\omega_{\varphi_\epsilon,h}}^2 \wedge \omega_{\varphi_\epsilon}^{n-1})=\sum_i \int_{\tilde{U}_i}  \frac{1}{N_i} \pi_i^* (\theta_i i \d \dbar  \log \vert \tau \vert_{\omega_{\varphi_\epsilon,h}}^2 \wedge \omega_{\varphi_\epsilon}^{n-1}).$$
Thus we have
$$\int_X c_1(T_X/\cE_i)\wedge \omega^{n-1}=\sum_i \int_{\tilde{U}_i}  \frac{1}{N_i} \pi_i^* (\theta_i i \d \dbar  \log \vert \tau \vert_{\omega_{\varphi_\epsilon,h}}^2 \wedge \omega_{\varphi_\epsilon}^{n-1}) +\int_X i \Theta(\det( T_X/\cE_i), h)\wedge \omega_{\varphi_\epsilon}^{n-1}$$
$$ \geq  - \epsilon\sum_i \int_{\tilde{U}_i}  \frac{1}{N_i} \pi_i^* (\theta_i \omega_{\varphi_\epsilon}^{n})=-\epsilon \int_X  \omega_{\varphi_\epsilon}^{n}=-\epsilon \int_X \omega^n.$$
The conclusion follows by taking $\epsilon \to 0+$
\end{proof}

Now we consider the general coefficients case following a suggestion of Mihai P\u{a}un.
\begin{mythm}
Let $(X, \omega)$ be an n-dimensional compact K\"ahler manifold.
Let $D=\sum (1-b_j) D_j=\sum (1-b_j) [s_j=0]$ a divisor with simple normal crossings with $b_j \in ]0,1[$ such that $-(K_X+D)$ is nef. 
Let
$$0 =\cE_0 \subset \cE_1 \subset \cdots \subset \cE_s=T_X$$ be a filtration of torsion-free subsheaves such that $\cE_{i+1}/\cE_i$ is an $\omega$-stable torsion-free subsheaf of $T_X/\cE_i$ of maximal slope. 
Then for any $i$, the slope of $\cE_{i+1}/\cE_i$ with respect to $\omega^{n-1}$,
namely
$$\mu(\cE_{i+1}/\cE_i):= \int_X c_1(\cE_{i+1}/\cE_i) \wedge \omega^{n-1},$$
is positive.
\end{mythm}
\begin{proof}
Construct as above proof of Theorem 2 a section $\tau \in H^0(X, \det(T_X/\cE_i) \otimes \Omega^r_X )$.
As in the above theorem 2, construct a sequence of metrics $\omega_{\varphi_\epsilon}$ with conic singularities by solving Monge-Ampère equations.
Consider local coordinate $(z^1, \cdots, z^n)$ such that the support of $D$ is given by $ \{ z^1 \cdots z^d =0 \} $.
In the general case, we do a cut-off in a tubular neighbourhood of $D$ of radius $\delta$ and study the asymptotic behavior as $\delta \to 0+$ to get the estimate of the slope from the integration of $i \d \dbar \log \vert \tau \vert_{\omega_{\varphi_\epsilon,h}}^2 \wedge \omega_{\varphi_\epsilon}^{n-1}$ over $X \setminus D$.
Locally
$$\tau=\sum_{I, |I|=r} \tau_I dz^I$$
with multi-index $I=\{i_1, \cdots, i_r\}$ of length $r$
and
$|\tau|^2_{\omega_{\varphi_\epsilon}}$ by the conic singularities is locally equivalent to
$\sum_{I, |I|=r} \prod_{j, 1 \leq i_j \leq d} |z_{i_j}|^{2-2b_{i_j}} |\tau_I|^2$.
Near the origin,
$$O(1) \leq -\log |\tau|^2_{\omega_{\varphi_\epsilon}} \leq -\log (\sum_{I, |I|=r} \prod_{j, 1 \leq i_j \leq d} |z_{i_j}|^{2} |\tau_I|^2)+O(1)$$
where $O(1)$ means a bounded term.
Thus to study the local integrability of $-\log |\tau|^2_{\omega_{\varphi_\epsilon}} $ with respect to some positive locally finite measure,
it is enough to study the local integrability of $-\log (\sum_{I, |I|=r} \prod_{j, 1 \leq i_j \leq d} |z_{i_j}|^{2} |\tau_I|^2)$
with respect to that measure
which has the same singularities as some local ideals.
Note that outside the support of $D$ (where the metric is smooth), pointwise, we have
$$i \d \dbar \log \vert \tau \vert_{\omega_{\varphi_\epsilon,h}}^2 \wedge  \omega_{\varphi_\epsilon}^{n-1} \geq  (- i \Theta(\det(T_X/\cE_i),h)- \epsilon \omega_{\varphi_\epsilon}) \wedge \omega_{\varphi_\epsilon}^{n-1}.$$

However, $-\log (\sum_{I, |I|=r} \prod_{j, 1 \leq i_j \leq d} |z_{i_j}|^{2} |\tau_I|^2)$ may have strict positive generic Lelong number along the support of $D$ which may not be integrable with respect to metric with Poincar\'e type singularities along $D$.
(For example, in the Poincar\'e type singularity case, the local integrability is equivalent to the fact that $1/(r |\log(r)|^\alpha)$ is locally integrable near 0 if and only if $\alpha >1$.)

Thus we need to rewrite $-\log (\sum_{I, |I|=r} \prod_{j, 1 \leq i_j \leq d} |z_{i_j}|^{2} |\tau_I|^2)$ in a golbal form and subtract the divisorial part along the support of $D$. 

Let $\omega_{\lceil D \rceil }$ be a smooth metric on $X \setminus D$ which is locally quasi-isometric near every point in $\lceil D \rceil=\{z_1 \cdots z_d=0\}$ to
$$\omega_{\lceil D \rceil}:= i\sum_{k=1}^d \frac{ dz_k \wedge d \bar z_k}{|z_k|^{2}} +i\sum_{k=d+1}^n dz_k \wedge d \bar z_k.  $$
(For example, this kind of metric can be constructed as follows. Cover $X$ by open coordinate charts such that $\lceil D \rceil$ can be written as zeros of local coordinates.
Assume furthermore that the coordinate charts on which $\lceil D \rceil$ is not irreducible do not intersect each other, up to taking some further refinement of the open cover.
Construct $\omega_{ \lceil D \rceil}$ by glueing local ones via a partition of unity.
Let us check that the glueing metric has desired properties.

Assume that $U_z$ is a coordinate chart such that $\lceil D \rceil \cap U_z=\{z_1 \cdots z_d =0\}$ and
that $U_w$ is a coordinate chart such that $\lceil D \rceil \cap U_w=\{w_1 =0\}$ with non-empty intersection.
Over $U_z \cap U_w$, $z_1 /w_1 \in \cO^{*}(U_z \cap U_w) $ since they define the same irreducible divisor on the intersection.
Thus for some $C>0$ large enough,
$i dz_1 \wedge d \overline{z}_1/ |z_1|^2+C i\sum_{k=1}^n dz_k \wedge d \bar z_k$ is equivalent to 
$i dw_1 \wedge d \overline{w}_1/ |w_1|^2+i\sum_{k=2}^n dw_k \wedge d \bar w_k$ over any chosen relative compact subset of $U_z \cap U_w$.
Thus after the partition of unity, the glueing metric has desired properties.
Of course, the glueing metric is not K\"ahler on $X$.)

Let $\cI$ be the ideal sheaf defined (locally) by the coefficients of type $\prod_{j, 1 \leq i_j \leq d} z_{i_j} \tau_I$ to which $\log |\tau|^2_{\omega_{\lceil D \rceil}}$ is equivalent to.
Note that $\cI$ is locally defined, but its integral closure is globally defined.
In fact, the germ of the integral closure is the holomorphic function germs $f$ such that locally
$$\log |f|^2 \leq \log |\tau|^2_{\omega_{\lceil D \rceil }}+O(1).$$
We denote this globally defined integral closure by $\overline{\cI}$.
Consider the divisorial valuation associated with $D_i$
$$\mu(\overline{\cI}, D_i):= \max \{m \in \N, \overline{ \cI} \subset \cI_{D_i}^m \}$$
where $\cI_{D_i}$ is the identification of $\cO(-D_i)$ as an ideal sheaf.
Consider 
$$\sigma:= \prod_i s_{D_i}^{\mu(\overline{\cI}, D_i)}$$
which is a global holomorphic section on $X$ where $s_{D_i}$ are the canonical sections of $\cO(D_i)$.
Consider $\log(\vert \tau \vert_{\omega_{\lceil D \rceil },h}^2/|\sigma|^2)$ such that any $D_i$ is not an irreducible component of its pole set.

Now outside the support of $D$ (where the metric is smooth), pointwise, we have
$$i \d \dbar  \log (\vert \tau \vert_{\omega_{\varphi_\epsilon,h}}^2/|\sigma|^2) \wedge   \omega_{\varphi_\epsilon}^{n-1} \geq  (-  i \Theta(\det(T_X/\cE_i),h)- \epsilon  \omega_{\varphi_\epsilon}+\sum_j i \mu(\overline{\cI}, D_j) \Theta(\cO_{X}(D_j))) \wedge  \omega_{\varphi_\epsilon}^{n-1}.$$

Note that by choice of $\sigma$, the poles of
$ \log(\vert \tau \vert_{\omega_{\lceil D \rceil },h}^2/|\sigma|^2)$
contain no irreducible component of $D$.
In local coordinates,
$ \log(\vert \tau \vert_{\omega_{\lceil D \rceil },h}^2/|\sigma|^2)$
has the same singularities as ideal quotient $\cJ:=(\cI: \prod_i \cI_{D_i}^{\mu(\overline{\cI}, D_i)})$.
In the following, we study the local integrability of $ \log(\vert \tau \vert_{\omega_{\lceil D \rceil },h}^2/|\sigma|^2)$ with respect to some volume form of mixed type of Poincar\'e type and conic type.

Let $\chi$ be a cut-off function $[0, \infty [ \to [0,1]$.
Following section 9 of \cite{CGP},
define
$$\rho(x):= \log (\log \frac{1}{ \prod_i |s_{D_i}|^2}).$$
Note that
$-i \d \dbar \rho$ is bounded from above by Poincar\'e type singularities along the support of $D$, which in local coordinates can be written as
$$ i\sum_{k=1}^d \frac{ dz_k \wedge d \bar z_k}{|z_k|^{2} \log^2(|z_k|)} +i\sum_{k=d+1}^n dz_k \wedge d \bar z_k. $$
We claim that
$$-\log(\vert \tau \vert_{\omega_{\lceil D \rceil },h}^2/|\sigma|^2) i \d \dbar \rho \wedge \omega_{\varphi_\epsilon}^{n-1} $$
is locally integrable.
It is enough to show that near the support of $D$ in local coordinates, for any $i_0 \in [1,d]$,
$$-\log(\vert \tau \vert_{\omega_{\lceil D \rceil },h}^2/|\sigma|^2) \frac{i d(z_{i_0}) \wedge d (\overline{z_{i_0}})}{|z_{i_0}|^2 \log^2 (|z_{i_0}|)} \wedge \prod_{i \neq i_0, i \leq d} \frac{i d(z_i) \wedge d (\overline{z_i})}{|z_i|^{2(1- \beta_i)} } \wedge  \prod_{i \geq d+1}  i dz_i \wedge d \overline{z_i}$$
is locally integrable.
To simplify the notations,
in the following, we assume that $i_0=1$.

Consider the branched cover
\[
\begin{array}{cccc}
\pi:& \mathbb{D}^d \times \D ^{n-d} & \longrightarrow & \mathbb{D}^d \times \D ^{n-d} \\
&(z_1, z_2, \ldots, z_d, z_{d+1}, \ldots, z_n) & \longmapsto & (z_1, z_2^{n_2}, \ldots, z_d^{n_d}, z_{d+1}, \ldots, z_n).
\end{array}
\]
with $n_i(\forall i \geq 2)$ large enough such that
$\beta_i n_i \geq 1$.
The local integrability is equivalent to the fact that
$$\pi^* \big( -\log(\vert \tau \vert_{\omega_{\lceil D \rceil },h}^2/|\sigma|^2) \frac{i d(z_{1}) \wedge d (\overline{z_{1}})}{|z_{1}|^2 \log^2 (|z_{1}|)} \wedge \prod_{d \geq i \geq 2} \frac{i d(z_i) \wedge d (\overline{z_i})}{|z_i|^{2(1- \beta_i)} } \wedge  \prod_{i \geq d+1}  i dz_i \wedge d \overline{z_i} \big)$$
is locally integrable.

Let $f_j(1 \leq j \leq r)$ be the local generators of $\pi^* \cJ$.
A sufficient condition is that
$$ -\log(\sum_j \vert f_j \vert^2) \frac{i d(z_{1}) \wedge d (\overline{z_{1}})}{|z_{1}|^2 \log^2 (|z_{1}|)} \wedge \prod_{i \geq 2} i d(z_i) \wedge d (\overline{z_i}) $$
is locally integrable with the choice of branched cover.

By the construction of $\cJ$,
the restrictions of $f_j$ on $\{z_1=0 \}$ are not all identically zero.
By continuity, on $\mathbb{D}_{r_0} \times \D ^{n-1}$
for small radius $r_0$,
the restrictions of $f_j$ on $\{z_1=s \} (\forall |s| < r_0)$ are not all identically zero.

Consider $-\log(-\log(\sum_j \vert f_j \vert^2))$
as a continuous family of psh functions on $\D^{n-1}$ parametrized by $z_1 \in \D_{r_0}$.
Note that the Lelong number of any psh function in this family at any point in $\D^{n-1}$ is 0.

By uniform Skoda integrability theorem (cf. e.g. Theorem 2.50 \cite{GZ17}), for any $s \in \D$ such that $|s| < r_0$,
$$\int_{\{s\} \times \D^{n-1}} -\log(\sum_j \vert f_j \vert^2)  \prod_{i \geq 2} i d(z_i) \wedge d (\overline{z_i}) \leq C$$
for some $C>0$ independent of $s$ by vanishing Lelong number.
Since the Poincar\'e metric in one variable is locally integrable,
we finish the proof of the claim about the local integrability by the Fubini theorem.

Similarly, by Schwarz inequalities, we can show that $$-\log(\vert \tau \vert_{\omega_{\lceil D \rceil },h}^2/|\sigma|^2) i \d  \rho \wedge \dbar \rho \wedge \omega_{\varphi_\epsilon}^{n-1}$$
is locally integrable.


For each  $\delta >0$, let $\theta_\delta : [0, \infty [ \to [0, 1]$ be a smooth function which is equal to zero on the
interval $[0, 1/\delta]$, and which is equal to 1 on the interval $[1 + 1/\delta, \infty]$. One may for example
define $\theta_\delta(x) = \chi(x - 1/\delta)$.
Consider $\chi_\delta: X \to [0,1]$ defined by
$$\chi_\delta(x):= 1- \theta_\delta(\rho(x)).$$
Consider Stokes's formula
$$\int_X \chi_\delta i \d \dbar (   \log (\vert \tau \vert_{\omega_{\varphi_\epsilon,h}}^2/|\sigma|^2) ) \wedge   \omega_{\varphi_\epsilon}^{n-1} =\int_X \d \chi_\delta \wedge i  \dbar  \log (\vert \tau \vert_{\omega_{\varphi_\epsilon,h}}^2/|\sigma|^2)  \wedge   \omega_{\varphi_\epsilon}^{n-1}
$$$$=-\int_X \dbar \chi_\delta \wedge i  \d  \log (\vert \tau \vert_{\omega_{\varphi_\epsilon,h}}^2/|\sigma|^2)  \wedge   \omega_{\varphi_\epsilon}^{n-1}=\int_X i \d \dbar \chi_\delta   \log (\vert \tau \vert_{\omega_{\varphi_\epsilon,h}}^2/|\sigma|^2)  \wedge   \omega_{\varphi_\epsilon}^{n-1}.$$
Note that the left-handed side term is pointwise bigger than $\chi_\delta (-  i \Theta(\det(T_X/\cE_i),h)- \epsilon  \omega_{\varphi_\epsilon}+\sum_j i \mu(\overline{\cI}, D_j) \Theta(\cO_{X}(D_j))) \wedge  \omega_{\varphi_\epsilon}^{n-1}$.
Moreover, the integration of the lower bound over $X$ tends to
$$\int_X (-c_1(T_X/\cE_i)-\epsilon \omega + \sum_j  \mu(\overline{\cI}, D_j) c_1(\cO(D_j))) \wedge \omega^{n-1}$$
as $\delta \to 0+$.

On the order hand, $i \d \dbar \chi_\delta= -\theta_\delta'(\rho)  i \d \dbar \rho- \theta_\delta''(\rho) i \d \rho \wedge \dbar \rho.$
Note that $|\theta_\delta''|,|\theta_\delta'|$ is uniformly bounded independent of $\delta$.
Thus
$\log (\vert \tau \vert_{\omega_{\varphi_\epsilon,h}}^2/|\sigma|^2) i \d \dbar \chi_\delta \wedge   \omega_{\varphi_\epsilon}^{n-1}$
is uniformly bounded by some integrable function by the previous claim.

By Lebesgue's dominated convergence, the limit of the right-handed term on the second line is 0 as $\delta \to 0+$.

Taking $\delta \to 0+$ gives 
$$0 \geq \int_X (-c_1(T_X/\cE_i)-\epsilon \omega + \sum_j  \mu(\overline{\cI}, D_j) c_1(\cO(D_j))) \wedge \omega^{n-1}$$
which finishes the proof by taking $\epsilon \to 0+$.
\end{proof}

Notice that a similar result is shown in Theorem 0.4 of \cite{Ou17} under the hypothesis that $-K_X$ is nef and that $X$ is projective with mild singularity. 

To illustrate the local integrability in the proof of Theorem 3, we give the following elementary example.
\begin{myex}
{\rm 
Take the same notations as in the proof of Theorem 3.
Assume that the complex dimension of $X$ is two and the support of $D$ has two irreducible components intersecting transversally at $x \in X$.
Assume that $\cJ$ is the maximal ideal at $x$.
Denote $(z_1,z_2)$ local coordinates near $x$ such that the support of $D$ is defined by $\{z_1 z_2=0\}$ where $x$ corresponds to the origin.

The local integrability near $x$ is equivalent to showing that in local coordinates, for any $i_0 \in [1,2]$,
$$-\log(\vert \tau \vert_{\omega_{\lceil D \rceil },h}^2/|\sigma|^2) \frac{i d(z_{i_0}) \wedge d (\overline{z_{i_0}})}{|z_{i_0}|^2 \log^2 (|z_{i_0}|)} \wedge \prod_{i \neq i_0, i \leq 2} \frac{i d(z_i) \wedge d (\overline{z_i})}{|z_i|^{2(1- \beta_i)} }$$
is locally integrable.

Let $\pi$ be the blow-up of $X$ at $x$. It is equivalent to showing that
$$\pi^* \big( -\log( |z_1|^2+|z_2|^2) \frac{i d(z_{i_0}) \wedge d (\overline{z_{i_0}})}{|z_{i_0}|^2 \log^2 (|z_{i_0}|)} \wedge \prod_{i \neq i_0, i \leq 2} \frac{i d(z_i) \wedge d (\overline{z_i})}{|z_i|^{2(1- \beta_i)} } \big)$$
is locally integrable near the exceptional divisor.
Without loss of generality, assume that $i_0=1$.
In local coordinates near $x$, $\pi$ is given by
$$\pi(w_1, w_2)=(w_1 w_s, \cdots, w_{s-1} w_s, w_s, w_{s+1}w_s, \cdots, w_{2} w_s).$$
In local coordinates, it is enough to show that
$$-\log|w_s|^2 \pi^*(\frac{i d(z_{1}) \wedge d (\overline{z_{1}})}{|z_{1}|^2 \log^2 (|z_{1}|)} \wedge  \frac{i d(z_2) \wedge d (\overline{z_2})}{|z_2|^{2(1- \beta_2)} } ) $$
is locally integrable near $\{w_s=0\}$.

If $s =2$, we have an upper bound for
$$-\log|w_2|^2  \frac{i d(w_1 w_2) \wedge d (\overline{w_1 w_2})}{|w_1 w_2|^2 \log^2 (|w_1 w_2|)} \wedge \frac{i d(w_2 ) \wedge d (\overline{w_2 })}{|w_2 |^{2(1- \beta_2)} }
$$
whose potential with respect to the Lebesgue measure is bounded by
$$-\log|w_2|^2  \times 1/(|w_1|^2 \log^2(w_1)) \times  1/ (|w_2|^{2(1- \beta_2)})$$
which is locally integrable.

If $s =1$, we have an upper bound for
$$-\log|w_1|^2  \frac{i d(w_1) \wedge d (\overline{w_1})}{|w_1|^2 \log^2 (|w_1|)} \wedge \frac{i d(w_1 w_2 ) \wedge d (\overline{w_1 w_2 })}{|w_1 w_2 |^{2(1- \beta_2)} }
$$
whose potential with respect to the Lebesgue measure is bounded by
$$-\log|w_1|^2  \times 1/(|w_1|^2 \log^2(w_1)) \times  |w_1|^{2 \beta_2}/ (|w_2|^{2(1- \beta_2)})$$
which is locally integrable.

}
\end{myex}

Now the arguments of Proposition 5.1 of \cite{Cao13} give the following corollary.
\begin{mycor}{\it
Let $(X, \omega)$ be an n-dimensional compact K\"ahler manifold.
Let $D=\sum (1-\beta_j) Y_j=\sum (1-\beta_j) [s_j=0]$ a divisor with simple normal crossings with $\beta_j \in ]0,1[$ such that $-(K_X+D)$ is nef. 
Then the Albanese morphism $\alpha_X$ is surjective with connected fibres.
In fact, the Albanese map is submersion outside an analytic set of codimension larger than 2.}
\end{mycor}
\begin{proof}
The proof in \cite{Cao13} only uses the fact that the slopes with respect to $\omega^{n-1}$ of the sheaves obtained as graded pieces of the Harder-Narasimhan filtration are positive.
Hence using Theorem 3, the result is a direct consequence of his arguments.
For the convenience of the readers, we just give here proof of the fact that the fibres of the Albanese map are connected.
We follow the arguments in the Proposition 3.9 of \cite{DPS94}.

Let $X \to Y \to \Alb(X)$ be the Stein decomposition of the Albanese map with $Y=\mathrm{Spec}\;  \alpha_{X*}\cO_X$.
Since $X$ is smooth, $Y$ is normal.
We claim that the map $f: Y \to \Alb(X)$ is \'{e}tale.
The reason is as follows.
By the arguments in \cite{Cao13}, there exists $Z$ an analytic subset in $\Alb(X)$ with codimension at least 2 such that $X\smallsetminus \alpha_X^{-1}(Z) \to \Alb(X)\smallsetminus Z$ is submersion (thus a fibration).
Thus $Y \smallsetminus f^{-1}(Z) \to \Alb(X) \smallsetminus Z$ is \'{e}tale.
We denote by $F$ the fibre of the fibration $f|_{Y \smallsetminus f^{-1}(Z) }$ which is finite.
By the long exact sequence associated with a fibration, we have
$$\pi_1(F) \to \pi_1(Y \smallsetminus f^{-1}(Z)) \to \pi_1(\Alb(X) \smallsetminus Z) \to \pi_0(F)$$
where $\pi_1(F)=0$ and $\pi_0(F)$ is finite.
In particular, $\pi_1(Y \smallsetminus f^{-1}(Z))$ is a free Abelian group of rank $2q:=2 \dim_{\C} \Alb(X)$.
Notice that by the codimension condition, we have $\pi_1(\Alb(X) \smallsetminus Z) \cong \pi_1(\Alb(X))$.
$\Alb(X)$ is isomorphic to the quotient of the universal cover $\C^q$ of  $\Alb(X) $ under the group action $\pi_1(\Alb(X) )$.
Define $T$ to be the quotient of $\C^q$ under the group action $\pi_1(Y \smallsetminus f^{-1}(Z) )$ with the natural cover $p: T \to \Alb(X)$.
By the homotopy lifting property, there exists a map $g: Y \smallsetminus f^{-1}(Z) \to T$ such that $p \circ g=f|_{Y \smallsetminus f^{-1}(Z)}$.
Remark that $g$ is holomorphic since it is given by the composition of $f$ with the holomorphic local inverse of $p$.
Since $Y \smallsetminus f^{-1}(Z) \to \Alb(X) \smallsetminus Z$ is finite, $f^{-1}(Z)$ is of codimension at least 2.
Since $Y$ is normal, $g$ extends to a morphism $g: Y \to T$.
Now $g$ is a generically injective morphism between $Y$ and $T$.
Since $T$ is smooth, the inverse map of $T \smallsetminus p^{-1}(Z) \to Y$ also extends across $p^{-1}(Z)$ which gives the inverse morphism of $g$.
In conclusion $g$ is a biholomorphism between $T$ and $Y$ which proves that $f$ is \'etale.
Another way to prove that $f$ is \'etale is using the purity of branch locus.
In fact, $f: Y  \to \Alb(X)$ is a ramified finite cover with $Y$ normal and $\Alb(X)$ smooth.
Then by the purity of the branch locus, the branch locus is either of codimension one or empty.
Since $f$ is a submersion outside an analytic set of $\Alb(X)$ of codimension at least 2, the branch locus has to be empty.

In particular, $Y$ is a finite \'{e}tale cover of the torus $\Alb(X)$,
so $Y$ itself is a torus.
By the universality of the Albanese morphism, there exists a morphism $h: \Alb(X) \to Y$ such that the morphism $X \to Y$ factorises through $h$.
Since the morphisms $X \to Y$ and $\alpha_X$ are surjective, we have $h \circ f=id_Y$ and $f \circ h=id_{\Alb(X)}$.
Thus $f$ is a biholomorphism, and the Albanese morphism has connected fibres.
\end{proof}
The arguments of \cite{Cao13} combined with Theorem 3 also give the following affirmation of a conjecture of Mumford.
The general conjecture of Mumford states that a projective or compact K\"ahler manifold $X$ is rationally connected if and only if $H^0(X, (T_X^*)^{\otimes m})=0$ for any $m \geq 1$.
\begin{mycor}{\it
Let $(X, \omega)$ be an n-dimensional compact K\"ahler manifold.
Let $D=\sum (1-\beta_j) Y_j=\sum (1-\beta_j) [s_j=0]$ a divisor with simple normal crossings with $\beta_j \in ]0,1[$ such that $-(K_X+D)$ is nef. 
Then the following properties are equivalent:

(1) $X$ is projective and rationally connected.

(2)$H^0(X, (T_X^*)^{\otimes m})=0$ for any $m \geq 1$.

(3) For every $m \geq 1$ and every finite \'etale cover $\tilde{X}$ of $X$, one has
$H^0(\tilde{X}, \Omega_{\tilde{X}}^{ m})=0$.}
\end{mycor}

\textbf{Acknowledgement} I thank Jean-Pierre Demailly, my PhD supervisor, for his guidance, patience and generosity. 
I would like to thank my post-doc mentor Mihai P\u{a}un, for many supports and very useful suggestions on this objective.
I would like to thank Junyan Cao, Henri Guenancia for some very interesting discussions on this objective.
This work is supported by 
DFG Projekt Singuläre hermitianische Metriken für Vektorbündel und Erweiterung kanonischer Abschnitte managed by Mihai P\u{a}un.
 
\end{document}